\documentclass[12pt]{article}

\usepackage{geometry}
\usepackage{amsmath,amssymb,mathtools}
\usepackage{amsthm}

\usepackage{tikz}

\usepackage{siunitx}
\usepackage{xspace}

\usepackage{float}

\usepackage{array}
\DeclarePairedDelimiter{\abs}{\lvert}{\rvert}
\newcommand{\card}[1]{\abs{#1}}

\input{my-symbols.tex}
\newcommand{\cglfuz}{\mathrel\fltrivb}

\newcommand{\cgup}{\mathord\uparrow}
\newcommand{\cgdown}{\mathord\downarrow}

\newcommand{\cgstar}{\mathord{\ast}}

\newcommand{\nimber}[1]{\ensuremath{\cgstar#1}}

\newcommand{\dos}{\mathbin{::}}

\newcommand{\cgslash}{\mathrel\vert}
\newcommand{\game}[2]{\{ #1 \mathrel\vert #2 \}}

\newcommand{\ruleset}[1]{\textup{\textsc{#1}}\xspace}

\newcommand{\ArcKayles}{\ruleset{Arc Kayles}}
\newcommand{\Col}{\ruleset{Col}}
\newcommand{\Domineering}{\ruleset{Domineering}}
\newcommand{\GenKonane}{\ruleset{Generalised Konane}}
\newcommand{\Hackenbush}{\ruleset{Hackenbush}}
\newcommand{\Konane}{\ruleset{Konane}}
\newcommand{\Nim}{\ruleset{Nim}}
\newcommand{\NodeKayles}{\ruleset{Node Kayles}}
\newcommand{\NoGo}{\ruleset{NoGo}}

\newcommand{\Poset}{\ruleset{Poset Games}}
\newcommand{\Snort}{\ruleset{Snort}}
\newcommand{\TT}{\ruleset{Teetering Towers}}

\newcommand{\DiPlace}{\ruleset{Digraph placement}}

\usepackage{hyperref}
\usepackage[capitalise, noabbrev]{cleveref}
\newtheorem{theorem}{Theorem}[section]
\newtheorem{conjecture}[theorem]{Conjecture}
\newtheorem{corollary}[theorem]{Corollary}
\newtheorem{lemma}[theorem]{Lemma}
\newtheorem{problem}[theorem]{Problem}
\newtheorem{question}[theorem]{Question}

\title{Digraph Placement Games}

\author{Alexander Clow\\
\small Department of Mathematics\\[-0.8ex]
\small Simon Fraser University\\[-0.8ex] 
\small Burnaby, British Columbia, Canada\\
\small\tt alexander\_clow@sfu.ca\\
\and
Neil Anderson McKay\\
\small Department of Mathematics and Statistics\\[-0.8ex]
\small University of New Brunswick\\[-0.8ex]
\small Saint John, New Brunswick, Canada\\
\small\tt neil.mckay@unb.ca\\}
\date{}

\begin{document}

\maketitle

\begin{abstract}
This paper considers a natural ruleset for playing a partisan combinatorial game on a directed graph, which we call \DiPlace. Given a digraph $G$ with a not necessarily proper $2$-coloring of $V(G)$, the \DiPlace game played on $G$ by the players Left and Right, who play alternately, is defined as follows. On her turn, Left chooses a blue vertex which is deleted along with all of its out-neighbours. On his turn Right chooses a red vertex, which is deleted along with all of its out-neighbours. A player loses if on their turn they cannot move. 
We show constructively that \DiPlace is a universal partisan ruleset; for all partisan combinatorial games $X$ there exists a \DiPlace game, $G$, such that $G = X$.
\DiPlace and many other games including \Nim, \Poset, \Col, \NodeKayles, \Domineering, and \ArcKayles are instances of a class of placement games that we call conflict placement games.
We prove that $X$ is a conflict placement game if and only if it has the same literal form as a \DiPlace game.
A corollary of this is that deciding the winner of a \DiPlace game is PSPACE-hard.
Next, for a game value $X$ we prove bounds on the order of a smallest \DiPlace game $G$ such that $G = X$.
\end{abstract}

\section{Introduction}

Let $G = (V,E)$ a directed graph and $\phi\colon V \to \{\text{blue},\text{red}\}$ a $2$-coloring of the vertices of $G$. Note that $\phi$ need not be a proper coloring. The \DiPlace game played on $G$, denoted simply by $G$ when the fact that we are playing a \DiPlace game is clear from context, is described as follows: Suppose there are two players, Left (blue) and Right (red). On her turn Left chooses a blue vertex $v$ of $G$ that has not been deleted, and deletes all vertices in $N^+[v]$. On his turn Right chooses a red vertex $u$ of $G$ that has not been deleted, and deletes all vertices in $N^+[u]$. If at the beginning of her turn, all blue vertices have been deleted, then Left loses. Similarly, if at the beginning of his turn, all red vertices have been deleted, then Right loses. 
When necessary we denote the \DiPlace game played on digraph $G = (V,E)$ with $2$-coloring $\phi$ by $G=(V,E,\phi)$.
The set of all \DiPlace games is the ruleset \DiPlace. 

Given a \DiPlace game $G$ and a sequence of vertices $u_1,\dots, u_t$ such that $u_i$ is not an out-neighbour of any vertex $u_j$ where $j<i$, we denote the game resulting from deleting vertices $u_1,\dots, u_t$ and their out-neighbourhoods in $G$ by $G/[u_1,\dots, u_t]$. To make for easier to read expressions, we denote $G/[u]$ simply as $G/u$ when only a single vertex and its out-neighbours are removed. We make no assumptions about the colours of removed vertices to allow for non-alternating play.

For readers familiar with the combinatorial game \NodeKayles, observe that \DiPlace can be seen as a directed version of partisan \NodeKayles. For readers unfamiliar, partisan \NodeKayles is played on a $2$-colored simple graph, where players take turns choosing a vertex of their color, call it $u$, then deleting the closed neighbourhood of $u$, $N[u]$, from the graph.

Any \DiPlace game is a two-player game with perfect information,  no chance, and where a player loses if they cannot move on their turn. As a result, every \DiPlace game is an example of a \emph{normal play combinatorial game}. The players in a \DiPlace game necessarily have different options on their move, such games are called \emph{partisan}, which simply implies that players need not have all of the same moves.  See \cref{fig: Small DiPlace game Example} for two examples of \DiPlace games. The first example is disconnected and  equal to a sum of simple combinatorial games, the second is connected and is a equal to the same sum of games.

\begin{figure}
\begin{minipage}[t]{0.49\textwidth}
\begin{center}
\scalebox{0.75}{
  \begin{tikzpicture}
  [node distance={15mm}, thick, main/.style = {draw, circle,}] 
\node[main][fill= blue] (d1) at (1,0) {}; 
\node[rectangle,draw,inner sep=5][fill= red] (d2) at (1,3) {}; 
\node[rectangle,draw,inner sep=5][fill= red] (d3) at (3,0) {};
\node[main][fill= blue] (d4) at (3,1.5) {}; 
\node[main][fill= blue] (d5) at (3,3) {};
\node[rectangle,draw,inner sep=5][fill= red] (d6) at (5,1.5) {}; 

\draw [line width=1.pt] (d4) -- (d5);
\draw [line width=1.pt] (d5) -- (d6);
\draw [line width=1.pt] (d4) -- (d3);
\draw [line width=1.pt] (d3) -- (d6);
\draw [line width=1.pt] (d4) -- (d6);
\draw [line width=1.pt] (d1) -- (d2);
\draw [line width=1.pt] (d3) -- (d2);
\draw [line width=1.pt] (d1) -- (d5);
\draw [line width=1.pt] (d1) -- (d6);
\draw [line width=1.pt] (d2) -- (d6);
\draw [line width=1.pt] (d3) -- (d2);
\draw [line width=1.pt][->] (d1) -- (d4);
\draw [line width=1.pt][->] (d2) -- (d4);
\draw [line width=1.pt][->] (d5) -- (d2);
\draw [line width=1.pt][->] (d3) -- (d1);

\node[main][fill= blue] (1) at (3,4.5) {}; 
\node[main][fill= blue] (2) at (1,4.5) {}; 

\node[main][fill= blue] (*1) at (3,6) {}; 
\node[rectangle,draw,inner sep=5][fill= red] (*2) at (1,6) {}; 
\draw [line width=1.pt] (*1) -- (*2);

\draw [line width=3.pt](-1,-1) -- (6,-1) -- (6,7) -- (-1,7) -- (-1,-1);

    \end{tikzpicture}
}
\end{center}
\end{minipage}
\begin{minipage}[t]{0.49\textwidth}
\begin{center}
\scalebox{0.75}{
\begin{tikzpicture}[node distance={15mm}, thick, main/.style = {draw, circle,}] 

\node[main][fill= blue] (1) at (1,0) {}; 
\node[main][fill= blue] (2) at (4,0) {}; 
\node[rectangle,draw,inner sep=5][fill= red] (3) at (5,3) {};
\node[main][fill= blue] (4) at (0,3) {}; 
\node[rectangle,draw,inner sep=5][fill= red] (5) at (2.5,6) {};

\draw [line width=1.pt] (3) -- (4);
\draw [line width=1.pt][->] (1) -- (3);
\draw [line width=1.pt][->] (1) -- (4);
\draw [line width=1.pt][->] (1) -- (5);
\draw [line width=1.pt][->] (2) -- (3);
\draw [line width=1.pt][->] (2) -- (4);
\draw [line width=1.pt][->] (2) -- (5);
\draw [line width=1.pt][->] (3) -- (5);
\draw [line width=1.pt][->] (4) -- (5);

\draw [line width=3.pt](-1,-1) -- (6,-1) -- (6,7) -- (-1,7) -- (-1,-1);
    \end{tikzpicture}

}
\end{center}
\end{minipage}
\caption{Two \DiPlace games equal to $2  + \cgdown + \cgstar$. Here $2  + \cgdown + \cgstar$ is chosen because it is an examples of simple, yet non-trivial combinatorial game. In this, and all other figures in this paper, an undirected edge $(u,v)$ implies that there exists directed edges $(u,v)$ and $(v,u)$. Blue vertices are drawn as circles and red vertices are drawn as squares.
}
\label{fig: Small DiPlace game Example}
\end{figure}

This paper is intended for a general mathematical audience. We assume familiarity with basic notation from graph theory, such as the in or out neighbourhood of a vertex, and induced subgraphs. We do not assume the audience is familiar with combinatorial game theory. We follow many conventions from combinatorial game theory, see  \cite{albert2019lessons,BerleCG1982,siegel2013combinatorial}.

A \emph{combinatorial game} is a two-player game which involves no randomness or hidden information. Normally the players in a combinatorial game are denoted Left and Right.
For technical reasons, a specific instance of a game is called a game, whereas the set of all games played with the same rules is called a ruleset.
Examples of combinatorial game rulesets include Chess, Checkers, and Go. A combinatorial game is \emph{short} if there exists an integer $k$, such that after $k$ turns the game must end, no matter the choices by either player, and at any point during the game, both players have at most a finite number of moves. This prevents infinite sequences of play or infinite numbers of followers. Games which allow infinite sequences of play or allow a player to have infinitely many moves on their turn are called \emph{long}. A combinatorial game is played under the \emph{normal} play convention if a player loses if and only if they cannot move on their turn. In this paper all the combinatorial games we consider are normal play and short. Hence, if we say that $X$ is a game, suppose that $X$ is short.

Given a game $X$ we let $L(X)$ be the set of games Left can move to if they move first in $X$, and we let $R(X)$ be the set of games Right can move to if they move first. We say a member of $L(X)$ is a \emph{Left option} of $X$ and a member of $R(X)$ is a \emph{Right option} of $X$. Often a game $X$ will be written as $\game{Y_1,\dots, Y_k}{Z_1,\dots,Z_t}$ where $L(X) = \{Y_1,\dots, Y_k\}$ and $R(X) = \{Z_1,\dots, Z_t\}$.

Given a game $X$ written as $\game{Y_1,\dots, Y_k}{Z_1,\dots,Z_t}$ and a game $X'$ written as \newline $\game{Y'_1,\dots, Y'_k}{Z'_1,\dots,Z'_t}$ we say that $X \cong X'$ if and only if $Y_i \cong Y'_i$ and $Z_j \cong Z'_j$ for all $i$ and $j$. The game with no moves from either player, $\game{}{}$, is denoted by $0$. So $0 \cong \game{}{}$.  If $X\cong X'$ then we say that $X$ and $X'$ have the same \emph{literal form}. Notice that literal forms are useful as they allow one to say that two games appearing in different rulesets are literally the same, even when the rulesets are dissimilar 

Games having the same literal form is too strong a notion of equivalence to be useful in many contexts. To that end we define another standard notion of equivalence in games, called equality. Given games $X $ and $Y$, the game $X+Y$ is the game with $L(X+Y) = \{X^L + Y: X^L \in L(X)\} \cup \{X + Y^L: Y^L \in L(Y)\}$  and $R(X+Y) = \{X^R + Y: X^R \in L(X)\} \cup \{X + Y^R: Y^R \in R(Y)\}$. This is often less formally written as
\[
X+Y = \game{X^L+Y, X+Y^L}{X^R+Y,X+Y^R}.
\]
We say that $X+Y$ is the \emph{disjunctive sum} of $X$ and $Y$.
Although it may not be immediately clear, the game $X+Y$ is the game played by each player choosing either $X$ or $Y$, but not both, to make a move in on their turn. As players lose when they have no move this new game played on two smaller games is well defined. 
Given a game $X$ we say the \emph{outcome} of $X$ is Left win if Left will win moving first or second, is Right win if Right will win moving first or second, is next player win if whoever moves next will win, or is previous player win if whoever moves first will lose. 
Now, we say games $Y$ and $Z$ are \emph{equal}, written as $Y=Z$, if and only if for all games $X$, the outcome of $X+Y$ is the same as the outcome of $X+Z$.

Observe that if $G$ is a \DiPlace game, whose underlying graph is disconnected, then $G$ is a disjunctive sum of its components. 
Also, notice that even if $G$ is a connected \DiPlace game, then $G$ can easily become disconnected through players deleting vertices.
Hence, understanding disjunctive sums, and therefore values, is critical to understanding \DiPlace games.

We say the class of games equal to $X$, is the \emph{value of $X$}. For each game $X$, the games with value $X$ have a unique simplest representative, which is called the \emph{canonical form of $X$}. For a proof of this see Theorem~2.7 and Theorem~2.9 from Chapter \MakeUppercase{\romannumeral 2} of \cite{siegel2013combinatorial}.

Modulo equality, the set of all games under this addition operation forms an abelian group. This group, and other algebraic objects like it, are the central object of study in combinatorial game theory. For more on the foundations of combinatorial game theory we recommend \cite{albert2019lessons,BerleCG1982,siegel2013combinatorial}.

\emph{Placement games}, also called \emph{medium placement games}, are a class of rulesets introduced in \cite{brown2019note}, satisfying the following properties,
\begin{enumerate}
    \item the game begins with an empty game board, and
    \item a move is to place pieces on one, or more, parts of the board subject to the rules of the game, and 
    \item the rules must imply that if a piece can be legally placed in a certain position on the board, then it was legal to place this piece in that position at any time earlier in the game, and
    \item once played, a piece remains on the board, never being moved or removed.
\end{enumerate}
Note here that deleting a vertex $v$ in \DiPlace can be viewed as placing a piece onto $v$, with the added rule that if an in-neighbour of a vertex $u$, or $u$ itself, has a piece placed on it, then no further piece may be placed onto $u$. In this way, the set of vertices with pieces placed on them at the end of the game will form a minimal directed dominating set.
So \DiPlace is a placement game. 

Placement games that satisfy the additional constraint, that if a position can be reached by a sequence of legal moves, then any sequence of moves which reach this position is legal, are called \emph{strong placement games}. This class of games was introduced by Faridi, Huntemann, and Nowakowski in \cite{faridi2019games1}, while the theory of strong placement games has been expanded upon in \cite{faridi2019games2,huntemann2018class,huntemann2019game,huntemann2021counting}. Notice that \DiPlace is not a strong placement ruleset.

Placement games include many well studied rulesets such as \Nim \cite{bouton1901nim}, \Hackenbush \cite{BerleCG1982,Conwa1976a}, \Col \cite{Conwa1976a,fenner2015game}, \Snort \cite{BerleCG1982,huntemann2021bounding,huntemann2024SNORT,schaefer1978complexity}, \NoGo \cite{brown2019note, chou2011revisiting}, \Poset \cite{clow2021advances,fenner2015game, fenner2022complexity, grier2013deciding,soltys2011complexity}, \NodeKayles \cite{bodlaender2002kayles,brown2019note,guignard2009compound}, \ArcKayles \cite{burke2024complexity,huggan2016polynomial,schaefer1978complexity}, \Domineering \cite{berlekamp1988blockbusting,breuker2000solving, huntemann2021counting, uiterwijk201611}, and \TT \cite{clow2023ordinal}.  
Note that throughout the paper, when discussing rulesets such as \Hackenbush and \Poset, we mean blue, green, red \Hackenbush and \Poset.
\Domineering is an example of a strong placement game while \Hackenbush is an example of a placement game which is not a strong placement game. 

These rulesets are diverse in notable ways. For example, one can calculate the winner of a game of \Nim or \TT in polynomial time \cite{bouton1901nim,clow2023ordinal}, whereas in other rulesets such as \Snort, \Col, and \Poset calculating the winner of a given position is PSPACE-complete \cite{fenner2015game,grier2013deciding,schaefer1978complexity}. As another example of the diversity of possible rulesets that are presentable as placement games, observe that rulesets such as $\Nim$ both players always have the same option, while in other rulesets such as \Domineering, players never the same options.

Given a ruleset $R$, one of the natural questions to ask is, for which games $X$ does there exist a game $G$ in $R$ where $G=X$. For example, it was shown by Conway \cite{Conwa1976a} that every game of \Col is equal to $x$ or $x+\cgstar$ for a number $x$. Alternatively,  Huntemann \cite{huntemann2018class} showed that many other kinds of values such a nimbers and switches appear in strong placement games. For many rulesets $R$, there exists a game $X$ such that for all games $G$ in $R$, $G \neq X$. For example, if $R$ is an impartial ruleset, then no game in $R$ is equal to $1 \cong \game{0}{}$.
If for all $X$ there exists a game $G$ in $R$ such that $G=X$, then we say that the ruleset $R$ is \emph{universal}.

Surprisingly, no ruleset was known to be universal until 2019. 
Even more interestingly, when a ruleset was discovered to be universal by
Carvalho and Pereira dos Santos in \cite{carvalho2019nontrivial}, this ruleset, called \GenKonane, was well established in the literature \cite{ernst1995playing,hearn2005amazons,nowakowski2002more} and did not have a mathematical origin. Deciding a winner in \Konane, the game which inspires \GenKonane, is known to be PSPACE-complete \cite{hearn2005amazons} and the game originates from the ancestral peoples of Hawaii.

Since 2022, several more universal rulesets have been discovered.
The first of these was given by Suetsugu \cite{Suets2022a} who showed another ruleset, called \ruleset{Turning Tiles}, was universal. Unlike \Konane, \ruleset{Turning Tiles} is not an established ruleset, having been first defined in \cite{Suets2022a}. Determining the winner in \ruleset{Turning Tiles} is also PSPACE-complete \cite{yoshiwatari2023turning}. The next year, in 2023, Suetsugu \cite{Suets2023} defined two more rulesets, called \ruleset{Go on Lattice} and \ruleset{Beyond the Door}, which they showed were also universal. By reducing \ruleset{Go on Lattice} and \ruleset{Beyond The Door} to \ruleset{Turning Tiles} it was also shown that determining the winner of a game of \ruleset{Go on Lattice} or \ruleset{Beyond the Door} is also PSPACE-complete \cite{yoshiwatari2023turning}.

The primary contribution of this paper is to show that \DiPlace is a universal ruleset. 
This is significant because
\DiPlace is the first placement ruleset which has been shown to be universal.
This is particularly notable because
\DiPlace, and placement games in general, provide an opportunity to build a new bridge between combinatorial game theory and other areas of mathematics, in particular graph theory.
This is not only because \DiPlace is played on a graph, as other rulesets, for example \ruleset{Go on Lattice}, can be generalised to be played on general graphs.
Rather,
placement games are much closer to traditional problems in graph theory, such as graph coloring and domination, compared to \GenKonane or the rulesets introduced by Suetsugu.
For example \Col is a game where players take turns proper coloring vertices of a graph until no vertices remain that can be properly colored. 
As another example, in \NodeKayles players take turns picking vertices, so that the sets of all chosen vertices form an independent set.

Like other placement games played on graphs such as \Col, \Snort, and \NodeKayles, many tools from graph theory can be applied to study \DiPlace games. Given \DiPlace is universal, this may allow future research to study the space of all short games in new ways.
Furthermore, there is also the possibility of tools from combinatorial game theory being helpful in graph theory.

The rest of the paper is structured as follows. \cref{Section:Prelim} is devoted to introducing definitions and more necessary considerations. This includes standard notation from combinatorial game theory. Next, in \cref{Section:Disguise} we prove that many well known rulesets are special cases of \DiPlace. \cref{Section:Universal} extends this discussion by proving that \DiPlace is a universal ruleset. This should be viewed as the main contribution of our paper. In \cref{Section:Extremal} we consider the related extremal question of given a game $X$, what is the smallest integer $n$ such that there is a \DiPlace game $G$ on $n$ vertices, where $G = X$? We conclude with a discussion of future work.

\section{Preliminaries}\label{Section:Prelim}

We must introduce some standard notation and ideas from the combinatorial game theory literature. We require these tools for proving various statements throughout the paper. What follows will be familiar to combinatorial game theorists. Such readers can proceed directly to \cref{Section:Disguise}.

As noted in the introduction, the set of all normal play games modulo equality forms an abelian group under the disjunctive sum operation. However, we have yet to address what an additive inverse of a game $X$ is. We clarify this now. Given a game $X = \game{Y_1,\dots, Y_k}{Z_1,\dots, Z_t}$, we define the \emph{negative of $X$}, denoted $-X$, as follows,
$$
-X = \game{-Z_1,\dots, -Z_t}{-Y_1,\dots, -Y_k}.
$$
Notice that $-X$ is essentially $X$, except the roles of each player are reversed. Hence, if $G$ is a \DiPlace game, then $-G$ is the \DiPlace game played on the same digraph, where blue vertices in $G$ are red in $-G$, and red vertices in $G$ are blue in $-G$.

Also, note that $0 \cong \game{}{}$ is the additive identity, due to the fact that $X+0\cong X$, and any games with the same literal forms are equal. Hence, we note that $X+(-X)$, which can be written as $X-X$ is equal to $0$. Notice that whoever moves first in $X-X$ and in $0$ loses. This is no coincidence as $X-X + 0 \cong X-X$ must have the same outcome as $0 + 0 \cong 0$ by the definition of disjunctive sum. We further note that $X-Y$ is previous player win if and only if $X-Y = 0$ if and only if $X = Y$.

Using this observation we extend our equivalence relation, equality, into a partial order. Given games $X$ and $Y$, we say $X\leq Y$ if Right wins moving second in $X-Y$. Equivalently, $X< Y$  if and only if Right wins $X-Y$ moving first or second. Symmetrically, we $X \geq Y$ if Left wins $X-Y$ moving second and $X> Y$ if and only if Left wins $X-Y$ moving first or second. We note that if $X \leq Y$ and $X\geq Y$, then this implies $X=Y$. 
We use $X \cglfuz Y$ to denote $X \not\geq Y$.
This implies, that $X$ is Left win if and only if $X > 0$, that $X$ is right win if and only if $X< 0$, that $X$ is previous player win if and only if $X = 0$, and $X$ is next player win if and only if $X \not\geq 0$ and $X \not\leq 0$.

Finally, we observe that there are many pairs of games $X$ and $Y$ such that $X \not\geq Y$ and  $X \not\leq Y$. For example the game $\game{1}{-1} \cong \game{\game{0}{}}{\game{}{0}}$ satisfies $\game{1}{-1} \not\geq 0$ and  $\game{1}{-1} \not\leq 0$. We also note that $\game{1}{-1} \not\geq 1 \cong \game{0}{}$ and  $\game{1}{-1} \not\leq 1 \cong \game{0}{}$. Notice that $X \not\geq Y$ and  $X \not\leq Y$ if and only if $X-Y$ is next player win.

These observations provide a very useful way of proving how the value of two games $X$ and $Y$ are related.
That is, if we want to check if $X\leq Y$, then we simply play $X-Y$ and determine if Left loses moving first.
Such a proof, is called a playing proof.
Moreover, these inequalities provide a means to study when a game $X$ is better for one player, than a game $Y$. If $X$ is at least as good for Left than $Y$, then $X \geq Y$.

These inequalities are important in defining \emph{dominated options} and numbers. Given a game $X = \game{Y_1,Y_2,\dots, Y_k}{Z_1,Z_2\dots, Z_t}$, we say a Left option $Y_1$ is dominated if there exists another Left option, call it $Y_2$, such that $Y_1 \leq Y_2$. Here we can think of $Y_2$ being a better move for Left than $Y_1$. Because $Y_2$ is a better move, we can ignore $Y_1$, since Left will never move to $Y_1$ when using an optimal strategy. Thus, when $Y_1$ is a dominated Left option, we can remove $Y_1$ without changing the value of $X$. That is,
\[
X = \game{Y_2,\dots, Y_k}{Z_1,Z_2,\dots, Z_t}.
\]
Similarly, if $Z_1 \geq Z_2$, then we say that $Z_1$ is a dominated Right option. As before, we can remove a dominated option without changing the value of $X$,
\[
X = \game{Y_1,Y_2,\dots, Y_k}{Z_2,\dots, Z_t}.
\]
It is important to note that canonical forms have no dominated options for either player.

The second way we will use these inequalities outside of playing proofs, is by considering a notable class of games called \emph{numbers}. We say that a game $X$ is a number if and only if all $X^L \in L(X)$ and all $X^R \in R(X)$ satisfy $X^L < X^R$, and all followers of $X$ are numbers. We point out that it was shown in \cite{Conwa1976a} that every number $x$ is equal to a \Hackenbush stalk, that is a blue, red path in \Hackenbush with a single ground edge. We also note that this class of games in called numbers, because if we allow for long games, then the field of real numbers injects into the set of number games, which form a field using the disjunctive sum as addition and a multiplication operation we do not define here. For more on this see \cite{Conwa1976a}. The set of short numbers is equal to the image of the dyadic rationals under this map. Given a number $X$, we say we often write $X  = x$, if the image of this map sends the real number $x$ to the value of $X$, and we use $x$ to denote the canonical form of $X$. For example, $1 \cong \game{0}{}$ and $\frac{1}{2} \cong \game{0}{1}$. We note that integers, that is games $X = n$ where $n$ is an integer, are the unique games in which one or both players may have no option.

The final definition we must give in this preliminaries section is a game's birthday. Given a game $X$, we let the \emph{birthday of $X$} be the greatest integer $b$, such that there exists a sequence of moves beginning in $X$ of length $b$. If $X$ is a short game, then $b$ is finite, if $X$ is a long game then $b$ is an ordinal. The canonical form of a game $X$ is the unique game equal to $X$ with smallest birthday.

\section{Digraph Placement Games in Disguise}\label{Section:Disguise}
 
Whereas  \DiPlace is first defined in this paper, we show that many placement games from well studied rulesets are instances of \DiPlace. For a formal statement of this see Theorem~\ref{Theorem: Conflict Placement equiv Digraph}. This result is a natural analogue to Theorem~4.4 from \cite{faridi2019games1} which shows that every strong placement game is an instance of a position in a given ruleset played on a simplicial complex.

Let $R$ be a placement ruleset and $X$ a game in $R$. We can record all moves in $X$ as triples $(x,T,p)$ where $T$ is the type of piece to be placed, $x$ is the location of the piece on the game board, and $p$ is the player who can legally make this move if it is their turn. If Left and Right can both place piece $T$ at location $x$ on their turn, then we record two triples: $(x,T,\text{Left})$ and $(x,T,\text{Right})$, respectively. 
For a placement game $X$, we let $M(X)$ denote the moves of $X$ expressed as triples.
For any follower $Y$ of $X$, $M(Y) \subseteq M(X)$ by the definition of a placement game. The options of $Y$ are not necessarily a subset of the options of $X$.

Now we define a class of placement games which we show characterize the literal forms of \DiPlace.
A placement ruleset $R$ is a \emph{conflict placement ruleset} if for all games $X$ in $R$, and all moves $(x,T,p)$ in $X$ there exists a set of moves in $X$, $A_{x,T,p}$, such that $(x,T,p)$ is a legal move in a follower $Y$ of $X$ if and only if $Y$ is obtained from $X$ by a sequence of legal moves $(x_1,T_1,p_1),\dots, (x_n,T_n,p_n)$ and 
\[
A_{x,T,p} \cap \{(x_1,T_1,p_1),\dots, (x_n,T_n,p_n)\} = \emptyset.
\]
If $X$ is a game in a conflict placement ruleset, then we say  $X$ is a \emph{conflict placement game}. In a conflict placement game $X$ and a move $(x,T,p)$ in $X$, we call the set $A_{x,T,p}$ the \emph{conflict set} of $(x,T,p)$.

\begin{figure}[ht!]
\centering
\includegraphics[scale = 1.0]{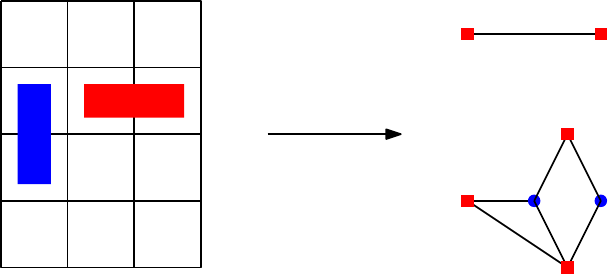}
\caption{An example of how a \Domineering position can be represented as a conflict placement game. For each red domino it is possible to play on the  board on the left there is a corresponding red vertex in the graph on the right side. For each blue domino it is possible to play on the board on the left side  there is a corresponding blue vertex in the graph on the right side. Vertices on the right side represent the triples $(x,T,p)$. Vertices $u$ and $v$ are adjacent if and only if they would overlap on the left  side; in \Domineering these overlaps generate conflict sets since the only thing that prevents a domino being placed is an already placed domino that overlaps it.}
\label{fig:Dom-is-conflict example}
\end{figure}

Notice that if $X$ is a short conflict placement game, then for all $(x,T,p) \in M(X)$, $(x,T,p)$ is in its own conflict set. Otherwise, the player $p$ can make the move $(x,T,p)$ an arbitrary number of times, which contradicts our assumption that $X$ is short. Hence, if a short conflict placement ruleset $R$ allows for a player to place a piece in the same position as an already placed piece, then the two pieces must be different types, otherwise this would allow an infinite sequence of play. Recall that all games we consider are short.

Observe that the rulesets \Nim, \Col, \Snort, \Poset, \NodeKayles, \ArcKayles, \Domineering, and \TT are all conflict placement rulesets. 
See Figure \ref{fig:Dom-is-conflict example} for an example of a \Domineering position as a conflict placement game; notice this is also a \DiPlace position. For example, in \Domineering a move $(x,T,p)$ has conflict set $\{(x',T',p'): x\cap x' \neq \emptyset\}$, and in \Poset if we record moves as elements of the poset, a move $x$ has conflict set 
$\{y: y \leq x\}$. In particular, \DiPlace is a conflict placement ruleset. Given a \DiPlace game $G$, and a vertex $v$ in $G$, the conflict set of the move which deletes $v$, is the set of moves given by either player deleting a vertex in $N^-[v]$. We will prove that every conflict placement game has the same literal form as a \DiPlace game. 

Familiar placement rulesets \Hackenbush and \NoGo are not conflict placement rulesets. For each ruleset there exist games where a move $(x,T,p)$ can be illegal after some set of moves $S$, where $\card{S} \geq 2$, are played but if a subset of the moves in $S$ are played then $x$ remains legal. For the reader who is familiar with the ruleset \Hackenbush see \cref{fig:non-conflict example} for an example of a game in \Hackenbush which is not a conflict placement game.

\begin{figure}[ht!]
\centering
\begin{tikzpicture}
\tikzstyle{hackennode}=[draw,circle,inner sep=0,minimum size=4pt]
\tikzstyle{hackenline}=[line width=3pt]
    \draw[hackenline,black] (0,0) -- (4,0);
    \node[hackennode][fill= white] (1) at (2,0) {}; 
    \node[hackennode][fill= white] (2) at (0,2) {}; 
    \node[hackennode][fill= white] (3) at (4,2) {};

    \draw[hackenline,blue] (1) -- (2);
    \draw[hackenline,red] (1) -- (3);
    \draw[hackenline,blue] (2) -- (3);
\end{tikzpicture}
\caption{A game of \Hackenbush, with literal form $\game{1-1,-\frac{1}{2}}{\game{0,1}{}}$, that is not a conflict placement game. This game is not a conflict placement game because the top edge cannot be cut if both bottom edges are cut, but the top edge can be cut if at most one of the bottom edges is cut.}
\label{fig:non-conflict example}
\end{figure}

\begin{theorem}\label{Theorem: Conflict Placement equiv Digraph}
  The game $X$ is a conflict placement game if and only if there exists a \DiPlace game $G \cong X$.
\end{theorem}

\begin{proof}
As noted, \DiPlace is a conflict placement ruleset  and thus every \DiPlace game $G$ has the same literal form as a conflict placement game. Now suppose $X$ is a conflict placement game.

We define a \DiPlace game $G = (V,E,\phi)$ as follows. 
Let the vertices be $V = M(X)$. Next, let $((x,T,p),(x',T',p'))\in E$ if and only if $(x,T,p)$ is in the conflict set of $(x',T',p')$. Finally, let $\phi(x,T,p) = \text{blue}$ if and only if $p = \text{Left}$ and $\phi(x,T,p) = \text{red}$ otherwise.

We show that $G \cong X$. In particular, we make the stronger claim that if any legal sequence of moves $(x_1,T_1,p_1),\dots,(x_n,T_n,p_n)$ is made from $X$ resulting in a game $Y$, then 
\[
G/[(x_1,T_1,p_1),\dots,(x_n,T_n,p_n)] \cong Y.
\]
This claim implies that $G \cong X$, as the game resulting from an empty sequence of moves made from $X$ is $X$.

As $X$ is a short game there exists an integer $b$, such that for all legal sequences of moves $(x_1,T_1,p_1),\dots,(x_n,T_n,p_n)$ made in $X$, $n \leq b$. Let $(x_1,T_1,p_1),\dots,(x_n,T_n,p_n)$ be a fixed sequence of legal moves made in $X$. We proceed by induction on $m = b-n$.

If $m = 0$, then $n = b$ and by our choice of $b$, $Y \cong 0$. This implies that no further pieces can be placed in the game, which implies for every move $(x,T,p)$ in $X$, a move $(x',T',p')$ in the conflict set of $(x,T,p)$ has been played in $(x_1,T_1,p_1),\dots,(x_n,T_n,p_n)$.
By our definition of $G$, this implies that,
\[
V \subseteq \bigcup_{1\leq i \leq n} N^+[(x_n,T_n,p_n)].
\]
Hence, $G/[(x_1,T_1,p_1),\dots,(x_n,T_n,p_n)] \cong 0$. Thus, if $m=0$, then 
\[G/[(x_1,T_1,p_1),\dots,(x_n,T_n,p_n)] \cong 0\cong Y\] as required.

Suppose now that $m>0$ and for all $0 \leq M < m$, if $(x'_1,T'_1,p'_1),\dots,(x'_{N},T'_{N},p'_N)$ is a sequence of legal moves in $X$ which results in a position $Y$ and $M = b -N$, then 
\[
G/[(x'_1,T'_1,p'_1),\dots,(x'_N,T'_N,p'_N)] \cong Y.
\]
Recall that $M(Y) \subseteq M(X)$. Hence, every move in $Y$ can be written in the form $(x,T,p)$. 
If $Y$ has no moves, then by the same argument as in the $m=0$ case, \[G/[(x_1,T_1,p_1),\dots,(x_n,T_n,p_n)] \cong 0 \cong Y.\] 
Suppose then that $Y$ has at least one move.

Let $(x,T,p)$ be a fixed move in $Y$
and let $Y_{x,t,p}$ be the game obtained making the move $(x,T,p)$ in $Y$.
The sequence of moves $(x_1,T_1,p_1),\dots,(x_n,T_n,p_n),(x,T,p)$ is length $n+1$, so by
the induction hypothesis $G/[(x_1,T_1,p_1),\dots,(x_n,T_n,p_n),(x,T,p)] \cong Y_{x,t,p}$.
As $(x,T,p)$ was an move in $Y$ chosen without loss of generality, we conclude that the same result holds for all moves in $M(Y)$.
Hence, what remains to be shown is that player $p$ can delete vertex $(x,T,p)$, for any move $(x,T,p) \in M(Y)$, and that the vertex set of $G/[(x_1,T_1,p_1),\dots,(x_n,T_n,p_n)]$ is equal to $M(Y)$.

We begin by proving that for each move $(x,T,p)$ in $Y$, player $p$ deleting vertex $(x,T,p)$ is a legal move in $G/[(x_1,T_1,p_1),\dots,(x_n,T_n,p_n)]$. By the definition of $\phi$, if $(x,T,p)$ is a vertex in
$G/[(x_1,T_1,p_1),\dots,(x_n,T_n,p_n)]$, then player $p$ can delete vertex $(x,T,p)$ on their turn.
Then it is sufficient to prove that for each move $(x,T,p)$ in $Y$, a vertex $(x,T,p)$ exists in $G/[(x_1,T_1,p_1),\dots,(x_n,T_n,p_n)]$.
Thus it is sufficient to prove that the vertex set of $G/[(x_1,T_1,p_1),\dots,(x_n,T_n,p_n)]$ is equal to the set of moves in $Y$ stored as triples.

Suppose that the vertex set of $G/[(x_1,T_1,p_1),\dots,(x_n,T_n,p_n)]$ is not equal to $M(Y)$. 
Either there exists a move $(x,T,p)$ in $Y$ that is not a vertex of 
\[
G/[(x_1,T_1,p_1),\dots,(x_n,T_n,p_n)],
\] 
or there exists a vertex of $G/[(x_1,T_1,p_1),\dots,(x_n,T_n,p_n)]$, call it $v$, which is not equal to an move $(x,T,p)$ in $Y$. 
Notice that if $m = b$, then $n=0$, implying that $X=Y$. 
Hence, if $m=b$, then $M(Y) = M(X) = V$ by the definition of $G$.
Suppose then that $m < b$.

Suppose that for some move $(x,T,p)$ in $Y$, there is no vertex $(x,T,p)$. As $(x,T,p)$ is a move in $Y$, $(x,T,p)$ is a move in $X$. Then, $(x,T,p)$ is a vertex in $G$. As $(x,T,p)$ is not a vertex in
$G/[(x_1,T_1,p_1),\dots,(x_n,T_n,p_n)]$, it must be the case that
\[
(x,T,p) \in \bigcup_{1\leq i \leq n} N^+[(x_i,T_i,p_i)].
\]
But by the definition of $G$, this implies that there exists an integer $i$, such that $(x_i,T_i,p_i)$ is in the conflict set of $(x,T,p)$. This contradicts $(x,T,p)$ being an move in $Y$.

Suppose then that there exists a vertex of $G/[(x_1,T_1,p_1),\dots,(x_n,T_n,p_n)]$, call it $v$, which is not equal to a move $(x,T,p)$ in $Y$. As $G/[(x_1,T_1,p_1),\dots,(x_n,T_n,p_n)]$ is a subgraph of $G$, $v$ is a vertex in $G$. Hence, $v = (x',T',p')$, which is an move in $X$.
As $(x',T',p')$ is not an move in $Y$, there must exists a integer $i$ such that $(x_i,T_i,p_i)$ is in the conflict set of $(x',T',p')$. 
But this implies that $v = (x',T',p')$ is not a vertex in $G/[(x_1,T_1,p_1),\dots,(x_n,T_n,p_n)]$. 
As this is a contradiction, we observe that the vertex set of $G/[(x_1,T_1,p_1),\dots,(x_n,T_n,p_n)]$ is in fact equal to $M(Y)$.

Therefore, 
\[
G/[(x_1,T_1,p_1),\dots,(x_{n},T_{n},p_{n})] \cong Y
\]
which implies that $G\cong X$ as $Y$ was chosen without loss of generality.
\end{proof}

There are immediate consequences of this concerning the computational complexity of deciding a winner in a \DiPlace game.

\begin{corollary}
    Determining the winner of a \DiPlace game $G$ is PSPACE-hard in the number $n$ of vertices in $G$.
\end{corollary}

\begin{proof}
    It was shown in \cite{fenner2022complexity} that it is PSPACE-complete to decide the winner of \Poset $P$ in terms of the number of elements in $P$. By \cref{Theorem: Conflict Placement equiv Digraph} every game in \Poset has the same literal form as a game in \DiPlace. Furthermore, each vertex in a \Poset game $P$ will correspond to at most $2$ vertices in the \DiPlace game $G$ with the same literal form as $P$, while transforming $P$ into $G$ is polynomial time in the number of elements in $P$.  Hence, there exists games in \DiPlace where it is PSPACE-complete to decided the winner. This completes the proof.
\end{proof}

\section{Universality}\label{Section:Universal}

Despite there being literal forms such as the example in \cref{fig:non-conflict example} which do not appear as \DiPlace games,
in this section we show that \DiPlace is a universal ruleset. See \cref{Theorem: Every value appears in DiPlace games} for this result. Our proof is constructive and requires some preliminaries from the literature. Before proving \cref{Theorem: Every value appears in DiPlace games} we consider four relevant lemmas, the first two of which are familiar to those familiar with combinatorial game theory.

\begin{lemma}[Chapter \MakeUppercase{\romannumeral 2} Theorem~3.7 \cite{siegel2013combinatorial}]\label{Lemma: Archimedean Principle}
Let $G$ be a short game. Then there exists an integer $n$ such that $-n < G < n$.
\end{lemma}

\begin{lemma}[Chapter \MakeUppercase{\romannumeral 2} Theorem~3.21 \cite{siegel2013combinatorial}]\label{Lemma: Number Translation Theorem}
    Let $G = \game{Y_1,\dots,Y_k}{Z_1,\dots,Z_t}$ be a short game not equal to a number and $x$ a short number. Then 
    $$
    G + x = \game{Y_1 + x, \dots, Y_k+x}{Z_1+x,\dots, Z_t+x}.
    $$
\end{lemma}

The first lemma we must prove is that every short number is equal to some \DiPlace game. For completeness we state it formally as follows.

\begin{lemma}\label{Lemma: Numbers are digraph games}
    If $x$ is a short number, then there exists a \DiPlace game $G = x$.
\end{lemma}

\begin{proof}
    By \cref{Theorem: Conflict Placement equiv Digraph} every literal form appearing in \Poset appears as a literal form in \DiPlace. As every \Hackenbush game played on a path with a single ground edge has the same literal form as a position in \Poset, every such \Hackenbush game has the same literal form as a \DiPlace game. It was shown by Conway \cite{Conwa1976a} that every short number $x$ is equal to a \Hackenbush game played on a path with a single ground edge. Hence, for all numbers $x$, there exists a \DiPlace game equal to $x$. This completes the proof.
\end{proof}

The next step in the proof of the main theorem requires a complicated construction (gadget).
Our construction is recursive. 
Let's first explore a relatively simple case.
Given \DiPlace games $G$ and $H$ we aim to construct a \DiPlace game equal to $\game{G}{H}$. Hence, we take a copy of $G$, a copy $H$, and some additional vertices and edges.
To get a Left option to $G$ we add a blue vertex which has an edge to every vertex not in our copy of $G$ and to get an option to $H$ we add a red vertex with an edge to every vertex not in our copy of $H$.
To ensure that the moves in either the copy of $G$ or the copy $H$ do not affect the value we include additional red and blue vertices that ensure dominated incentives for these moves.
The inclusion of these extra vertices (and some edges) results in a game with value $\game{-1,G}{1,H}$.
We will argue in the proof of \cref{Theorem: Every value appears in DiPlace games} that the existence of such a \DiPlace game is sufficient to imply the existence of a \DiPlace game equal to $\game{G}{H}$. Our formal construction is given for games with multiple options for each player. See \cref{Lemma: Gadget Works Lemma} for the value of what we construct. 

For fixed but arbitrary positive integers $k$ and $t$, let $G_1,\dots,G_k$ and $H_1,\dots,H_t$ be \DiPlace games. Let $n$ be a positive integer. We define a new \DiPlace game denoted by $n\langle G_1,\dots, G_k \cgslash H_1,\dots, H_t \rangle$; this game is played on the directed graph $G$. The vertex set of $G$, $V(G)$, is the disjoint union of the vertex sets of the digraphs $G_1,\dots, G_k$ and $H_1,\dots, H_t$ and vertex sets $\{b_1,\dots, b_k\}, \{r_1,\dots, r_t\}$, and $\{x_1,\dots, x_{n}\},\{y_1,\dots, y_{n}\}$. 
As a slight abuse of notation we will write $G_i$ for $V(G_i)$ and
$H_j$ for $V(H_j)$.
The $2$-coloring of $G$, $\phi$, is as follows;
\begin{itemize}
    \item all vertices $v$ in $G_i$ or $H_j$ receive the same colours as they did in $G_i$ or $H_j$, and
    \item for all $1 \leq i \leq k$, $\phi(b_i) = \text{blue}$, and
    \item for all $1 \leq j \leq t$, $\phi(r_j) = \text{red}$, and
    \item for all $1 \leq i \leq n$, $\phi(x_i) = \text{blue}$, and
    \item for all $1 \leq j \leq n$, $\phi(y_j) = \text{red}$.
\end{itemize}
We define the edge set of $G$ as follows. Notice that the resulting digraph will have positive density 
(i.e.~a positive proportion of all arcs are present).
Hence, the description of how vertices are connected is relatively complex. 
To assist the reader we provide, in Figure~\ref{fig: Construction Example},
a simple example of the construction and a technical explanation of the moves in said example.
\begin{enumerate}
    \item for all $i$, $G_i$ induces a copy of $G_i$, and
    \item for all $j$, $H_j$ induces a copy of $H_j$, and
    \item $\{x_1,\dots, x_n\} \cup \{y_1,\dots, y_n\}$ is an independent set, and
    \item  $\{b_1,\dots, b_k\} \cup \{r_1,\dots, y_t\}$ induces a digraph with all possible arcs, and
    \item for all $i$, $b_i$ has an arc to every vertex not in $G_i$, and
    \item for all $j$, $r_j$ has an arc to every vertex not in $H_j$, and    
    \item for any pair $i\neq j$, all possible arcs exist between $G_i$ and $G_j$, and
    \item for any pair $i\neq j$, all possible arcs exist between $H_i$ and $H_j$, and   
    \item for any pair $i\neq j$, all possible arcs exist between $G_i$ and $H_j$, and
    \item for all blue vertices $v$ not in $\{x_1,\dots, x_n\}$, $v$ has an arc to all vertices in $\{x_1,\dots, x_n\}$, and
    \item for all red vertices $v$ not in $\{y_1,\dots, y_n\}$, $v$ has an arc to all vertices in $\{y_1,\dots, y_n\}$, and
    \item for all $i$, $x_i$ has an arc to every vertex not in $\{x_1,\dots, x_n\} \cup \{y_1,\dots, y_n\}$, and
    \item for all $i$, $y_i$ has an arc to every vertex not in $\{x_1,\dots, x_n\} \cup \{y_1,\dots, y_n\}$. 
\end{enumerate}

This describes the \DiPlace game $n\langle G_1,\dots, G_k \cgslash H_1,\dots, H_t \rangle$. Given sets of \DiPlace games $L = \{G_1,\dots, G_k\}$ and $R = \{H_1,\dots, H_t\}$, we let $n\langle L\cgslash R\rangle = n\langle G_1,\dots, G_k \cgslash H_1,\dots, H_t \rangle$.

\begin{figure}[ht]
\begin{center}
\includegraphics[scale = 1.0]{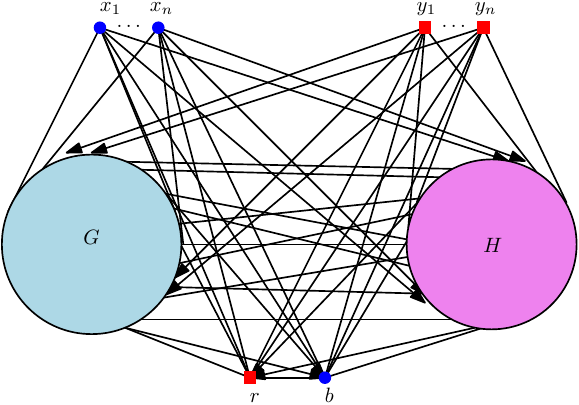}
\end{center}
\caption{A figure displaying the \DiPlace game $n\langle G | H\rangle$.}
\label{fig: Construction Example}
\end{figure}

Since this construction is somewhat opaque on a first reading, and it is critical to the rest of this section,
we do our best to describe the literal form of $n\langle G_1,\dots,G_k|H_1,\dots, H_t\rangle$.
To start we consider $n\langle G|H \rangle $, pictured in Figure~\ref{fig: Construction Example}.
To aid this description we will introduce several other operations, as they assist with bookkeeping.
For a positive integer $n$ and game $G$
we let $n \dos G$ be defined by 
\[
n \dos G = \game{n-1,G^L}{n\dos G^R}.
\]
Similarly, we define 
\[
-n \dos G = \game{-n\dos G^L}{1-n, G^R}.
\]
Now, in $n\langle G|H \rangle $ the Left move in $b$ is a move from $n\langle G|H \rangle$ to $G$.
Left moving in $\{x_1,\dots, x_n\}$ yields a graphs equal to $-1$, with $n-1$ blue vertices, $n$ red vertices, and no edges.
All other Left moves are in the copy of $G$ or the copy of $H$.
The reader can verify that if Left moves in the copy of $G$ to $G^L$, then this moves $n\langle G|H\rangle$ to some $-n\dos G^L$.
By the symmetry of our construction, this means
\[
n\langle G|H \rangle  = \game{G,-n \dos G^L, -n \dos H^L, -1}{1,n \dos G^R, n \dos H^R,H}.
\]
Once the reader is satisfied with this, it is not hard  to see how this form generalizes to $n\langle G_1,\dots,G_k|H_1,\dots, H_t \rangle$

We now proceed to \cref{Lemma: Gadget Works Lemma}.

\begin{lemma}\label{Lemma: Gadget Works Lemma}
    Let $G_1,\dots, G_k$ and $H_1,\dots, H_t$ be \DiPlace games. Then there exists a Digraph Placement game 
    \[G = \game{-1,G_1,\dots, G_k}{1, H_1,\dots, H_t}.\]
\end{lemma}

\begin{proof}
    As $G_1,\dots, G_k$ and $H_1,\dots, H_t$ are finite lists of short games, the set $S$ of all games $Q$ such that $Q$ is an option of some game $G_i$ or $H_j$, or the option of an option of some game $G_i$ to $H_j$, is finite. Then by \cref{Lemma: Archimedean Principle} there exists a positive integer $n$ such that for all $Q \in S$,
    \begin{equation*}
    -n < Q < n.
    \end{equation*}
    Observe that any $n$ greater than the order of all graphs $G_1,\dots, G_k$ and $H_1,\dots, H_t$ satisfies the above property.
    Let $n$ be such an integer and let $G \cong (2n+2)\langle G_1,\dots, G_k | H_1,\dots, H_t \rangle$. We claim that 
    \[
    G = \game{-1,G_1,\dots, G_k}{1, H_1,\dots, H_t}.
    \]
    We prove that Left has options to every value in the set $\{-1,G_1,\dots, G_k\}$, and every non-dominated Left option takes a value in the set $\{-1,G_1,\dots, G_k\}$; we also prove that Right has an option to every value in the set  $\{1, H_1,\dots, H_t\}$, and every non-dominated Right option takes a value in the set 
    $\{1, H_1,\dots, H_t\}$. 
    We prove this fact for Left's options; the result for Right's options follows by a similar argument due to the symmetry of our construction.

    Suppose it is Left's turn.
    Let $1\leq i \leq k$ be fixed.
    Observe that $N^+[b_i] = V(G) \setminus V(G_i)$.
    Hence, if Left deletes vertex $b_i$, then all vertices in $G$ except those vertices in $G_i$ are also deleted.
    As $G[V(G_i)] \cong G_i$, Left has an option to $G_i$.
    As $1\leq i \leq k$ was chosen without loss of generality, Left has an option to $G_i$ for all $1\leq i\leq k$.
    
    Observe that for all $1\leq i \leq 2n+2$, 
    $$
    N^+(x_i) = V(G) \setminus (\{x_1,\dots,x_{2n+2}\} \cup \{y_1,\dots,y_{2n+2}\}).
    $$
    Hence, if Left deletes any blue vertex $x_i$, then this leaves a \DiPlace game with $2n+1$ blue vertices, $2n+2$ red vertices, and  no edges. Such a game is equal to $-1$. Hence, Left has an option to $-1$.

    The remaining Left options in $G$ to consider are any moves which delete a blue vertex in $G_i$ for some $1\leq i \leq k$ or $H_j$ for some $1 \leq j \leq t$.
    Let $1\leq i\leq k$ and $1 \leq j \leq t$ be fixed.  
    Notice that if $v \in V(G_i)$, then 
    $$
    V(G) \setminus N^+[v] = (V(G_i) \setminus N^+[v]) \cup \{y_1,\dots,y_{2n+2} \}
    $$
    and that if  $v \in V(H_j)$, then 
    $$
    V(G) \setminus N^+[v] = (V(H_j) \setminus N^+[v]) \cup \{y_1,\dots,y_{2n+2} \}.
    $$ 

    We claim that by our choice of $n$, deleting any blue vertex $v \in V(G_i)$ or $v \in V(H_j)$ is a dominated Left option. Suppose without loss of generality $v \in V(G_i)$. To prove this we will show that 
    \[
    G/v < - 1.
    \]
    In order to show this we will prove that
    \[
    G/v < (G_i/v) - n - 1.
    \]
    Notice that $(G_i/v) - n - 1< -1$ by our choice of $n$. 

    Consider the difference $G/v - G_i/v+n+1$. 
    Notice that $y_1$ is a red vertex and 
    \[
    G/[v,y_1] \cong -2n-1.
    \] 
    Hence, Right moving first by deleting $y_1$ moves the difference to a game equal to $G_i/v - n$, which is negative by our choice of $n$.
    So Right wins $G/v - G_i/v+n+1$ moving first.
    
    Next, we show that all of Left's moves in the difference $G/v - G_i/v+n+1$ are losing. 
    If Left moves $n+1$ to $n$ in the difference, then Right can reply in $y_1$ resulting in a game equal to 
    \[
    -G_i/v - n-1 < -G_i/v - n < 0.
    \]
    since $-n < Q$ for all followers $Q$ of $G$.
    Hence, Left loses moving in the $n+1$ summand.
    As a result, we note that if Left has a winning option in $G/v - G_i/v+n+1$, then Left has a winning option by moving in $G/v$ or by moving in $G_i/v$. 
    Notice that no blue vertex in $G/v$ has a edge pointing to any vertex in the set $\{y_{1},\dots, 7_{2n+2}\}$.
    If Left deletes vertex $u \in V(G/v)$, then Right wins by deleting $y_1$ on their turn because as before this results in a game equal to $G_i/v - n$, which is negative.
    If Left deletes vertex $u \in V(G_i/v)$, then this leaves the game $G/v - G_i/[v,u] + n +1$. Again, Right deletes $y_1$ in response leaving the game 
    $$
    -2n -1 - G_i/[v,u] +n +1 = -G_i/[v,u] - n
    $$
    which is negative since $-n < Q$ for all followers $Q$ of $G$. Hence, Left loses moving first. This implies the desired result that $G/v < -1$. As $v$ was chosen without loss of generality, Left deleting a blue vertex $v \in V(G_i)$ is a dominated Left option for all blue vertices $v \in V(G_i)$ for all $1\leq i \leq k$. Notice that the same argument can be made when $v \in V(H_j)$ is a blue vertex. 

    It follows that Left has options to every value in the set $\{-1,G_1,\dots, G_k\}$, while every non-dominated Left option takes a value in the set $\{-1,G_1,\dots, G_k\}$. As stated earlier, our claim that Right has an option to every value in the set $\{1, H_1,\dots, H_t\}$, while every non-dominated Right option takes a value in the set $\{1, H_1,\dots, H_t\}$ follows by a symmetric argument. This concludes the proof.
\end{proof}

With all of the necessary preliminaries complete we are prepared to prove that \DiPlace is a universal ruleset.

\begin{theorem}\label{Theorem: Every value appears in DiPlace games}
 If $X$ is a short game, then there exists a \DiPlace game $G$, such that $G = X$.
\end{theorem}

\begin{proof}
    Let $X$ be a short game. By \cref{Lemma: Numbers are digraph games} if $X$ is a short number, then there exists a $G = X$. Suppose then that $X$ is not a number. Then $X$ has a Left option and a Right option, given that the only games without options for both players are integers. 

    Suppose without loss of generality that $X$ is in canonical form and 
    $$
    X = \game{Y_1,\dots,Y_k}{Z_1,\dots,Z_t}.
    $$
    Every option of $X$ has a birthday smaller than $X$. As $X$ has an option for both players and $X$ is in canonical form, $X \neq 0$. Hence, the birthday of $X$, say  $b$, is strictly positive. Suppose for the sake of contradiction that $X$ is not equal to any \DiPlace game, and $b$ is the least integer such that a game $X$ with birthday $b$ not equal to any \DiPlace game.

    By this assumption for all $1\leq i\leq k$ and $1\leq j \leq t$, there exists a \DiPlace game $G_i = Y_i$ and a \DiPlace game $H_j = Z_j$. By \cref{Lemma: Gadget Works Lemma} there exists a \DiPlace game $G'$ which satisfies that
    \[
    G' = \game{-1,G_1,\dots, G_k}{1, H_1,\dots, H_t}.
    \]
    If $G' = X$, then we have contradicted our assumption that no \DiPlace game equals $X$. Suppose then that $G' \neq X$.

    If $G' \neq X$, then either $-1$ is not dominated by a Left option of $X$ or $1$ is not dominated by a Right option of $X$. Observe that at least one of these options must be dominated,  as otherwise, $X = 0$ contradicting our assumption that $X$ is not a number. Without loss of generality consider the case where $-1$ is dominated by a Left option of $X$ but $1$ is not dominated by a Right option of $X$.

    Now consider the difference, $G' - X$. Right wins going second by following a mirroring strategy. Similarly, Right loses moving first unless Right takes their option to $1$ in $G'$. By our assumption that $G'\neq X$, it follows that $1 - X \leq 0$, given that we have shown all other moves are losing for either player. Given our assumption that $X$ is not a number, this implies that $X > 1$. 

    We let $n$ be the smallest integer such that $X < n$, this is defined due to \cref{Lemma: Archimedean Principle}.
    As $X > 1$, we conclude that $n \geq 2$.
    Then, by the minimality of $n$, $X = n -1$ or $n -1 \cglfuz X$. But we have assumed that $X$ is not an integer, so $X \neq n - 1$ implying that $n -1 \cglfuz X$.
    Hence, $0 \cglfuz X - n + 1 < X$ and as a result $X - n + 1$ is either Left win, or next player win. 
    That is, Left going first has an option in $X - n + 1$ which is at least $0$ and so  dominates $-1$.
    We note for future consideration that $X - n + 1 < 1$ given $(X - n + 1) - 1 = X - n < 0$ by our choice of $n$.
    
    We claim that there exists a \DiPlace game $G'' = X - n + 1$. 
    By assumption $X$ is not a number, so
    \cref{Lemma: Number Translation Theorem} implies that
    $$
    X - n + 1 = \game{Y_1-n+1,\dots,Y_k-n+1}{Z_1-n+1,\dots,Z_t-n+1}.
    $$
    As $1-n$ is a number, there exists a \DiPlace game equal to $1-n$, hence, for all $1\leq i \leq k$ and $1\leq j \leq t$, there exists a \DiPlace game equal to $G_i - n +1 = Y_i - n + 1$ and a \DiPlace game equal to $H_j - n +1 = Z_j - n + 1$ formed by taking disjoint unions of \DiPlace games. Thus, by \cref{Lemma: Gadget Works Lemma} there exists a \DiPlace game $G''$ such that 
\begin{eqnarray*}
    G'' & = & \game{-1,G_1-n+1,\dots,G_k-n+1}{1,H_1-n+1,\dots,H_t-n+1}\\
    & = & \game{-1,Y_1-n+1,\dots,Y_k-n+1}{1,Z_1-n+1,\dots,Z_t-n+1}.
\end{eqnarray*}
    We claim that $G'' = X - n +1$.

    Recall that Left has an option in $X-n+1$ which dominates $-1$. So $-1$ is dominated as a Left option in $G''$. Thus, $G'' = H$ where 
 \begin{eqnarray*}
    H & = & \game{Y_1-n+1,\dots,Y_k-n+1}{1,Z_1-n+1,\dots,Z_t-n+1}.
\end{eqnarray*}   
    Consider the difference $H - (X-n+1)$. As before, a mirroring strategy implies that Left will loses moving first in this difference, while a mirroring strategy also implies that Right will lose moving first, unless they take their option to $1$ in $H$. Suppose that Right moves first and takes their option in $H$ to $1$. Then the resulting game is of the form
\begin{eqnarray*}
    1 - (X - n +1) & = & 1 - X + n  -1\\
    & = & n - X > 0
\end{eqnarray*}
    by our choice of $n$.
    So Right also loses if they move first in $H - (X - n +1)$. Hence, $H - (X-n+1)$ is previous player win, which implies that $G'' = H = X - n +1$. By \cref{Lemma: Numbers are digraph games} there exists a \DiPlace game equal to $n-1$, hence there exists a \DiPlace game, call it $G = G'' + n - 1 = (X - n+1) + n - 1 = X$ formed by taking the disjoint union of $G''$ and a \DiPlace game equal to $n-1$. As this contradicts our assumption that $X$ is not equal to a \DiPlace game, this assumption must be false, so $X$ is equal to a \DiPlace game.
    
    The case where $1$ is dominated by a Right option of $X$ but $-1$ is not dominated by a Left option of $X$ leads to a contradiction by a symmetric argument. Here symmetric refers to switching the roles of the players. Given $X$ was chosen to be a minimal counter-example, we conclude that there exists no counter-example. This complete the proof.  
\end{proof}

\section{Constructing Games Values in Smallest Graphs}\label{Section:Extremal}

Now that we know \DiPlace is a universal ruleset, as shown in \cref{Theorem: Every value appears in DiPlace games}, we can consider the matter of which values are hardest to construct in \DiPlace.
Given a game $X$, let $f(X)$ be the least integer such that there exists a \DiPlace game $G$ with $f(X)$ vertices where $G = X$.
In this section we provide bounds on the function $f(X)$ for general $X$.

We let $\mathbb{G}_b$ be the set of values born by day $b$, $g(b) = \card{\mathbb{G}_b}$, and we let 
\[
F(b) = \max_{X\in \mathbb{G}_b} f(X).
\]
Next let $a(b)$ be the cardinality of a largest subset $S \subseteq \mathbb{G}_b$ such that for all $X,Y \in S$, $X \not \leq Y$ and $X \not\geq Y$; that is, $a(b)$ is the cardinality of the largest \emph{antichain} of games born by day $b$. Observe that by containment, $a(b) \leq g(b)$.

Our upper bound comes from the construction used in \cref{Section:Universal}. Notably, even for a number of small games our construction does not minimize the number of vertices in a graph $G = X$. See \cref{Appendix:Day2}.
Do to this, we believe the upper bound we give here can be significantly improved.
Our lower bounds are derived from some standard combinatorial and graph theoretic arguments.

We begin by analysing the number of vertices used in the construction given to prove \cref{Theorem: Every value appears in DiPlace games}.

\begin{lemma}\label{Lemma: F(b) upper bound in terms of b and a(b)}
    Let $b\geq 0$.
    Then 
    \[
    F(b+1) \leq 2a(b)(F(b)+b+1) + 5b + 8.
    \]
\end{lemma}

\begin{proof}
    For an integer $n$, digraph games $G_1,\dots, G_k$, and $H_1,\dots, H_t$, the game \newline $n\langle G_1,\dots,G_k\cgslash H_1,\dots,H_t\rangle$ has exactly
    \[
    2n + k + t + \sum_{1 \leq i \leq k} \card{V(G_i)} + \sum_{1 \leq j \leq t} \card{V(H_j)}
    \]
    vertices. By \cref{Lemma: Gadget Works Lemma} and the proof of \cref{Theorem: Every value appears in DiPlace games}, for  all games of the form 
    $X = \game{Y_1,\dots, Y_k}{Z_1,\dots, Z_t}$, either
\begin{align*}
    X  & = (2m+2) \langle G_1,\dots, G_k\cgslash H_1,\dots, H_t \rangle
\end{align*}
or 
\begin{align*}
    X  & = (2m+2) \langle G_1-m+1,\dots, G_k-m+1\cgslash H_1 - m+1,\dots, H_t-m+1 \rangle + m - 1,
\end{align*}
    where for all options of $X$ and all options of options of $X$, $Q$, we have $-m < Q < m$.
    
    Here for all $1 \leq i \leq k$, $G_i = Y_i$ and for all $1\leq j \leq t$, $H_j = Z_j$
    where $G_1,\dots, G_k$ and $H_1,\dots, H_t$ are all \DiPlace games. Observe that for every option $X'$ of $X$, $f(X'-m+1) \leq f(X') + m -1 \leq F(b) + m - 1$. Assume without loss of generality that for all options $X'$ of $X$, we make $X-m+1$ as a \DiPlace game with $f(X')+m-1$ vertices.
    Then the second form of $X$ above contains more vertices than the first, so we can assume that $X$ is of the second form without loss of generality. 

    Suppose without loss of generality that $X = \game{Y_1,\dots, Y_k}{Z_1,\dots, Z_t}$ is born on day $b+1$ and is in canonical form. Then $X$ has no dominated options and every option of $X$ is born by day $b$. Hence, removing dominated options from $\{-1,G_1-m+1,\dots, G_k-m+1\}$ to form a set $L$ and removing dominated options from $\{-1,H_1-m+1,\dots, H_t-m+1\}$ to form a set $R$, notice that $X  = G \cong (2m+2)\langle L\cgslash R\rangle - m + 1$ and $\card{L},\card{R} \leq a(b)$. Furthermore, we note that for all games $Z$ born by day $b$, $-b-1 < Z < b+1$, so we may choose $m = b+1$. 

    Hence,
    \begin{align*}
        f(X) & \leq 4b + 8 + 2a(b) + a(b)(F(b) + b) + a(b)(F(b)+b) + b \\
        & = 2a(b)(F(b)+b+1) + 5b + 8.
    \end{align*}
    This concludes the proof.
\end{proof}

We have the following useful lower bound of $F(b)$.

\begin{lemma}[Lemma~9 \cite{wolfe2004counting}]
\label{Lemma: Lower bound on g(b)}
    For all $b\geq 0$, 
    \[
    g(b+1) \geq g(b)^2.
    \]
\end{lemma}

\begin{lemma}\label{Lemma: bounds on the number of non-isomorphic digraphs on n vertices}
Let $n$ be a fixed integer and let $G(n)$ equal the number of distinct game values  equal to a \DiPlace game on at most $n$ vertices. Then for all $n\geq 2$,
\[
\log_2(G(n)) \leq \log_2(n) + n^2.
\]
\end{lemma}

\begin{proof}
Observe that each \DiPlace game is equal to exactly one game value. Hence, $G(n)$ is at most the number of \DiPlace games on at most $n$ vertices. Observe that there are exactly 
$$
2^{2\binom{n}{2}} = 2^{n(n-1)}
$$
labeled directed graphs on $n$ vertices without repeated edges of the same orientation. Next, observe that for a fixed labeled directed graph $G$ on $n$ vertices, there are $2^n$, $2$-colorings of the vertices $G$. Hence, there are at most 
\[
2^n \cdot 2^{n(n-1)} = 2^{n^2}
\]
\DiPlace games with exactly $n$ vertices. 
Noting that the \DiPlace game with no vertices and the two vertex \DiPlace game with $1$ blue vertex, and $1$ red vertex, and no edges, are both equal to $0$, we conclude that for $n\geq 2$
\[
G(n) \leq \Big(1+\sum_{m=1}^n m2^{m^2}\Big)-1 \leq n2^{n^2}
\]
where the plus $1$ covers the \DiPlace game with no vertices, and the minus $1$ covers the double count of the at least two \DiPlace 
games on at most $n$ vertices 
having value $0$.
This implies that 
\[
\log_2(G(n)) \leq \log_2(n) + n^2
\]
as required.
\end{proof}

We now can and do bound $F(b)$ below by a function of $b$.

\begin{theorem}\label{Theorem: F(b) lower bound in birthday}
    For all $b \geq 2$,
    $$
    F(b) \geq \lfloor 2^{b/2} \rfloor.
    $$
\end{theorem}

\begin{proof}
    Let $F(b) = n$. Then $G(n) \geq g(b)$ where $G(n)$ is the number of distinct game values $X$ equal to a \DiPlace game on at most $n$ vertices.  
    We now consider a lower bound on the value of $g(b)$.

    The literal forms born by day $1$ are $0 \cong \game{}{}$, $1 \cong \game{0}{}$,  $-1 \cong \game{}{0}$, and $\cgstar \cong \game{0}{0}$. No two of these games are equal, so $g(1) = 4$. By induction, for $b\geq 2$  \cref{Lemma: Lower bound on g(b)} implies that
    \begin{align*}
        g(b) & \geq g(1)^{2^{b-1}} \\
        & = 4^{2^{b-1}}\\
        & = 2^{2^b}.
    \end{align*}
    Thus $G(n) \geq 2^{2^b}$ and furthermore $\log_2(G(n)) \geq 2^b$. \cref{Lemma: bounds on the number of non-isomorphic digraphs on n vertices} implies that
    \begin{align*}
        \log_2(n) + n^2 & \geq 2^b.
    \end{align*}
    Hence,  $n \geq \lfloor 2^{b/2} \rfloor$, since $\lfloor \sqrt{\log_2(n)+n^2} \rfloor = n$
    for all integers $n \ge 1$.
\end{proof}

Having now established a lower bound for $F(b)$, we turn our attention to getting specific bounds for $b \le 5$. For $b\leq 2$ we obtain exact values of $F(b)$ in this paper. 
The authors, along with Davies, obtain exact values for $F(3)$ in \cite{clow2025constructing}.
Since the work for $b = 3$ was completed after the submission of this paper, but before acceptance, we will not refer to it or its findings
further in this paper.
For $b\geq 4$ there remains significant room for improvement in future work.

Our estimations of $F(b)$ for small $b$ use several methods. For $b\leq 2$ we construct games by hand for our upper bound and use combinatorial arguments to prove that these are tight. For $b\geq 3$ we obtain our upper bounds by bootstrapping using \cref{Lemma: F(b) upper bound in terms of b and a(b)} and computational results from the literature regarding $a(b)$ and $g(b)$. Hence, we include some tables of small values of $a(b)$ and $g(b)$. Our lower bounds for $3 \leq b\leq 5$ are trivial and giving significant improvements seems a hard problem.

\begin{lemma}[Chapter \MakeUppercase{\romannumeral 3}, Section 1 \cite{siegel2013combinatorial}, Lemma~3, Theorem~2 \cite{Suets2022b}, Theorem~11 \cite{wolfe2004counting}]
\label{Lemma: ag small b}
    For $b \leq 4$, $g(b)$ and $a(b)$ satisfy,
\begin{center}
\begin{tabular}{|c |c |c |} 
 \hline
  $b$ & $a(b)$ & $g(b)$ \\ [0.5ex] 
 \hline\hline
 $0$ & $1$ & $1$ \\ 
 \hline
 $1$ & $2$ & $4$ \\
 \hline
 $2$ & $4$ & $22$  \\
 \hline
 $3$ & $86$ & $1474$  \\
 \hline 
 $4$ & $ < 4\cdot 10^{184}$ & $ < 4 \cdot 10^{184}$  \\ 
 \hline
\end{tabular}
\end{center}
\end{lemma}

\begin{theorem}\label{Theorem: F(012345)}
For $b \leq 5$, $F(b)$ satisfies,
\begin{center}
\begin{tabular}{|c |c | c |} 
 \hline
  $b$ & $ \leq F(b) $ & $ F(b) \leq $  \\ [0.5ex] 
 \hline\hline
 $0$ & $0$ & $0$ \\ 
 \hline
 $1$ & $2$ & $2$\\
 \hline
 $2$ & $4$ & $4$\\
 \hline
 $3$ & $5$ & $74$  \\
 \hline
 $4$ & $6$ & \num{13439} \\ 
 \hline
 $5$ & $7$ & $1.08 \cdot 10^{189}$\\
 \hline
\end{tabular}
\end{center}
\end{theorem}

\begin{proof}
    We divide the proof into cases depending on which value $b$ takes in $\{0,1,2,3,4,5\}$.

\vspace{0.5cm}
\noindent
\underline{Case.1:} $b = 0$.

The game $0 \cong \game{}{}$ is the unique game born on day $0$. As neither player has a move, the \DiPlace game played on an empty graph has value $0$. So $F(0) = 0$.

\vspace{0.5cm}
\noindent
\underline{Case.2:} $b = 1$.

The literal forms born by day $1$ are $0 \cong \game{}{}$, $1 \cong \game{0}{}$,  $-1 \cong \game{}{0}$, and $\cgstar \cong \game{0}{0}$. The games $1$ and $-1$
are equal to the \DiPlace game on a single blue vertex and a single red vertex respectfully. Meanwhile, $\cgstar$ is equal to the \DiPlace game with a single blue vertex $u$ and a single red vertex $v$, such that $(u,v),(u,v)$ are both edges. Hence, $F(1)\leq 2$.

To show that $F(1)\geq 2$, observe that $\cgstar$ has an option for both players. Hence, if $G$ is a \DiPlace game equal to $\cgstar$ there must be at least one blue vertex and at least one red vertex. Thus, $F(1)= 2$.

\vspace{0.5cm}
\noindent
\underline{Case.3:} $b = 2$.

By \cref{Lemma: ag small b}, $g(2) = 22$, and a full list of game values born by day $2$ is given in \cite{Suets2022b}. 
We construct all $22$ of these values born by day $2$ in  as \DiPlace games with at most $4$ vertices in \cref{Appendix:Day2}.
Hence, $F(2) \leq 4$.

Observe that if a \DiPlace game, say $H$, has no blue vertices, then $H$ is an integer and $H \le 0$ with equality holding exactly when $H$ is empty. Similarly, if $H$ has no red vertices, then $H$ is an integer and $H \ge 0$.

To see that $F(2) \geq 4$, consider the game $\game{1}{-1}$ and let $G$ be a \DiPlace game equal to $\game{1}{-1}$. 
Without loss of generality, consider Left moving first. If every Left option of $G$ has no blue vertices, then every option is a non-negative integer. Because Left wins $G$ going first, there must then be an option to the empty game. The option to the empty game is not reversible; hence $G \neq \game{1}{-1}$, a contradiction. That is, Left has an option to a game with a blue vertex and thus must itself have two blue vertices. Since Left was chosen without loss of generality, there must also be two red vertices. Thus, $F(2) = 4$.

\vspace{0.5cm}
\noindent
\underline{Case.4:} $b = 3$.

Observe that $F(3) \leq 74$ follows from \cref{Lemma: F(b) upper bound in terms of b and a(b)} given $F(2) = 4$ by case.3 and \cref{Lemma: ag small b}. To see that $F(3) \geq 5$, notice that there exist game values born on day $3$ with an option to a game equal to $\game{1}{-1}$. Hence, any game that has $\game{1}{-1}$ as a follower requires at least $5$ vertices.

\vspace{0.5cm}
\noindent
\underline{Case.5:} $b = 4$.

Observe that $F(4) \leq \num{13439}$ follows from \cref{Lemma: F(b) upper bound in terms of b and a(b)} given $F(3) \leq 74$ by case.4 and \cref{Lemma: ag small b}. It is immediate that $F(4) \geq F(3)+1 \geq 6$. Thus, $6 \leq F(4) \leq \num{13439}$.

\vspace{0.5cm}
\noindent
\underline{Case.6:} $b = 5$.

Observe that $F(5) \leq 1.08 \cdot 10^{189}$ follows from \cref{Lemma: F(b) upper bound in terms of b and a(b)} given $F(4) \leq  \num{13439}$ by case.5 and \cref{Lemma: ag small b}. It is immediate that $F(5) \geq F(4)+1 \geq 7$. Thus, $7 \leq F(5) \leq 1.08 \cdot 10^{189}$.
\end{proof}

We now can and do give an upper bound for $F(b)$ that is not recursive, unlike \cref{Lemma: F(b) upper bound in terms of b and a(b)}.
This bound is not the best possible upper bound achievable by these methods. However, in the authors' opinion, any improvement made primarily using these methods would remain far from tight. Hence, we give the following bound that is relatively simple to state.

\begin{theorem}\label{Theorem: F(b) upper bound in birthday}
    For all $b\geq 1$, 
    \[
    F(b) < \frac{1}{2}g(b+1) - b.
    \]
\end{theorem}

\begin{proof}
The result holds for $1 \leq b \leq 3$ by \cref{Lemma: ag small b} and by \cref{Theorem: F(012345)}.
Suppose that the result holds for $F(b)$ and consider $F(b+1)$ where $b\geq 3$.
Applying \cref{Lemma: F(b) upper bound in terms of b and a(b)} and the induction hypothesis, we see that 
\begin{align*}
    F(b+1) &\leq 2a(b)(F(b)+b+1) + 5b + 8 \\
    &\leq 2g(b)(F(b)+b+1) + 5b + 8. \\
    &\leq g(b)g(b+1)+ 2g(b)+ 5b+ 8. 
\end{align*}
Using \cref{Lemma: Lower bound on g(b)} twice gives $g(b+2) \geq g(b)^2 g(b+1)$. This implies that
\begin{align*}
    F(b+1) &\leq \frac{g(b+2)}{g(b)} + 2g(b) + 5b + 8.
\end{align*}
We claim that $\frac{g(b+2)}{g(b)} + 2g(bS) + 5b + 8 < \frac{g(b+2)}{2} - (b+1)$. To show this, we consider the difference
\[
\frac{g(b+2)}{2} - b - 1 - \left(\frac{g(b+2)}{g(b)} + 2g(b)+ 5b + 8\right).
\]
By \cref{Lemma: Lower bound on g(b)}, for $b\geq 2$,
\begin{align*}
        g(b) & \geq g(1)^{2^{b-1}} \\
        & = 4^{2^{b-1}}\\
        & = 2^{2^b}.
\end{align*}
For all $b \geq 3$, $g(b) \geq 2^{2^b} > 6b + 9$;
hence  
\[
\frac{g(b+2)}{2} - \left(\frac{g(b+2)}{g(b)} + 2g(b) + 6b + 9\right) > 0
\]
if and only if 
$$
g(b)g(b+2) - 2g(b+2) - 2(2g(b)+6b+9)g(b) >0.
$$

From here it is not hard to verify that
\begin{align*}
    g(b)g(b+2) - 2g(b+2) - 2(2g(b)+6b+9)g(b) & = (g(b)-2)g(b+2) - 2(2g(b)+6b+9)g(b) \\
    & \geq (g(b)-2)g(b+2) - 4g(b)^2 + 2(6b+9)g(b)  \\
    & \geq (g(b)-2)g(b+2) - 6g(b+1)\\
    & > 6g(b+2) - 6g(b+1) \\
    & = 6\Bigl( g(b+2) - g(b+1) \Bigr)\\
    &> 0.
\end{align*}
For all $b\geq 3$, $g(b) > 8$ and $g(b)$ is a strictly increasing monotone function. Thus, 
\begin{align*}
    F(b+1) &\leq \frac{g(b+2)}{g(b)} + 5b+9 \\
    & < \frac{g(b+2)}{2} - b
\end{align*}
as desired. This concludes the proof.
\end{proof}

\section{Future Work}\label{Section:Future}

Here we discuss some questions for future work. The first and perhaps most obvious of these is the following.
However, between submission and acceptance of this paper the authors and Davies \cite{clow2025constructing} resolved this problem 
by showing $F(3)=8$.

\begin{problem}
    Determine the value of $F(3)$.
\end{problem}

One can also ask how $F(b)$ grows asymptotically. 
From the limited data provided by small examples it seems that that $F(b)$ grows slowly compared to $g(b)$.
This leads into our first conjecture. 

\begin{conjecture}
    $F(b) = o(g(b))$.
\end{conjecture}

\cref{Theorem: F(012345)} implies that for $b\leq 4$, $F(b) < g(b)$. Moreover, we notice that the gap between $F(b)$ and $g(b)$ is expanding rapidly for $b\leq 4$. Hence, it seems reasonable to assume, given the nature of how fast these functions are growing, that these early gaps will further expand as $b$ tends to infinity, hence this conjecture. To prove this result, one would need a new approach to bounding $F(b)$ above, or, using \cref{Lemma: F(b) upper bound in terms of b and a(b)},  to prove that $a(b)F(b) = o(g(b))$. We observe that there is little to indicate that $a(b)F(b) = o(g(b))$, although the possibility of this cannot be discounted.

Another approach to understanding the construction of game values in \DiPlace is to maximize $f(X)$ over smaller sets of games $X$. What values $X$ are the hardest to make in \DiPlace? One way to formalize this is as \cref{fut:hardtofind}.

\begin{problem}\label{fut:hardtofind}
    Find a set of game values $S$, such that there exists a constant $c>0$, where if $S_b$ is the set of all games in $S$ born by day $b$,
    \[
    \min_{X\in S_b} f(X) \geq cF(b)
    \]
    for all sufficiently large $b$.
\end{problem}

Alternatively, which game values $X$ born by day $b$ require much fewer than $F(b)$ vertices to make? Numbers, switches, uptimals and other notable classes of games are all candidates. We note that every nimber $\nimber{n}$ can be constructed with $2n$ vertices through the use of ordinal sums. We conjecture that this is best possible.

\begin{conjecture}
Let $n\geq1$ be an integer. Then,
\[
f(\nimber{n})=2n.
\]
\end{conjecture}

A similar meta question is to ask which graph theoretic parameters of the underlying graph $G$ impact the winner and the value of a \DiPlace game $G$. 
There are too many such questions to list; we provide three examples.

\begin{question}
    Huntemann and Maciosowski \cite{huntemann2024SNORT} recently proved that the maximum degree of a graph $G$ and the temperature of the  \Snort game played on $G$ can be arbitrarily far apart. How do degree and temperature relate in \DiPlace games?
\end{question}

\begin{question}
Do there exist values $X$ such that $X$ is not equal to any connected \DiPlace game?
\end{question}

\begin{question}
    For which graph classes $\mathcal{G}$ can the winner of all \DiPlace games $G$ in $\mathcal{G}$ be solved in polynomial time?
\end{question}

Tools that would likely assist in answering the aforementioned problems are methods for helpfully simplifying a \DiPlace game $G$ while ensuring that the resulting \DiPlace game $H$ satisfies $G=H$. Do such operations exist?

In a similar vein one can ask, how do operations on graphs, which are natural from the perspective of graph theory, impact the values of \DiPlace games? The first author and Finbow conjectured in \cite{clow2021advances} that the value of a lexicographic product of a family of impartial \Poset can be written in closed form as a function of the values of its factors. 
For which families of \DiPlace games do useful properties of \Poset remain true?

\bibliographystyle{plain}
\bibliography{bib}

\appendix

\section{Minimal constructions of Values Born by Day 2}\label{Appendix:Day2}

All $22$ game values born by day $2$ are constructed as \DiPlace games. We use the minimum number of vertices required for each value. There are several values that require $4$ vertices.

\bigskip

\begin{center}
\begin{tikzpicture}[node distance={15mm}, thick, main/.style = {draw, circle,}] 

\node[fill=none] at (-2,0) (nodes) {$-2$~:};

\draw [line width=3.pt](-1,-1) -- (3,-1) -- (3,1) -- (-1,1) -- (-1,-1);

\node[rectangle,draw,inner sep=5][fill= red] (1) at (0,0) {}; 
\node[rectangle,draw,inner sep=5][fill= red] (2) at (2,0) {};

\node[fill=none] at (5,0) (nodes) {$2$~:};

\draw [line width=3.pt](6,-1) -- (10,-1) -- (10,1) -- (6,1) -- (6,-1);

\node[main][fill= blue] (1) at (7,0) {}; 
\node[main][fill= blue] (2) at (9,0) {}; 

\end{tikzpicture}
\end{center}

\vspace{0.75cm}
\begin{center}
\begin{tikzpicture}[node distance={15mm}, thick, main/.style = {draw, circle,}] 

\node[fill=none] at (-2,0) (nodes) {$-1$~:};
\draw [line width=3.pt](-1,-1) -- (1,-1) -- (1,1) -- (-1,1) -- (-1,-1);

\node[rectangle,draw,inner sep=5][fill= red] (1) at (0,0) {};

\node[fill=none] at (5,0) (nodes) {$1$~:};
\draw [line width=3.pt](6,-1) -- (8,-1) -- (8,1) -- (6,1) -- (6,-1);

\node[main][fill= blue] (1) at (7,0) {}; 

\end{tikzpicture}
\end{center}

\vspace{0.75cm}
\begin{center}
\begin{tikzpicture}[node distance={15mm}, thick, main/.style = {draw, circle,}] 

\node[fill=none] at (-2,1) (nodes) {$-1 + \cgstar$~:};
\draw [line width=3.pt](-1,-1) -- (3,-1) -- (3,3) -- (-1,3) -- (-1,-1);

\node[rectangle,draw,inner sep=5][fill= red] (1) at (0,0) {}; 

\node[main][fill= blue] (2) at (0,2) {}; 
\node[rectangle,draw,inner sep=5][fill= red] (3) at (2,1) {}; 

\draw [line width=1.pt] (1) -- (2);

\node[fill=none] at (5,1) (nodes) {$1 + \cgstar$~:};
\draw [line width=3.pt](6,-1) -- (10,-1) -- (10,3) -- (6,3) -- (6,-1);

\node[rectangle,draw,inner sep=5][fill= red] (1) at (7,0) {}; 

\node[main][fill= blue] (2) at (7,2) {}; 
\node[main][fill= blue] (3) at (9,1) {}; 

\draw [line width=1.pt] (1) -- (2);

\end{tikzpicture}
\end{center}

\vspace{0.75cm}
\begin{center}
\begin{tikzpicture}[node distance={15mm}, thick, main/.style = {draw, circle,}] 

\node[fill=none] at (-2.25,1) (nodes) {$-\frac{1}{2}$~:};
\draw [line width=3.pt](-1,-1) -- (3,-1) -- (3,3) -- (-1,3) -- (-1,-1);

\node[main][fill= blue] (1) at (0,1) {}; 

\node[rectangle,draw,inner sep=5][fill= red] (2) at (2,1) {}; 

\draw [line width=1.pt] [->] (2) -- (1);

\node[fill=none] at (4.75,1) (nodes) {$\frac{1}{2}$~:};
\draw [line width=3.pt](6,-1) -- (10,-1) -- (10,3) -- (6,3) -- (6,-1);

\node[rectangle,draw,inner sep=5][fill= red] (1) at (7,1) {}; 

\node[main][fill= blue] (2) at (9,1) {}; 

\draw [line width=1.pt] [->] (2) -- (1);
\end{tikzpicture}
\end{center}

\vspace{0.75cm}
\begin{center}
\begin{tikzpicture}[node distance={15mm}, thick, main/.style = {draw, circle,}] 

\node[fill=none] at (-2.5,1) (nodes) {$\game{\cgstar}{-1}$~:};
\draw [line width=3.pt](-1,-1) -- (3,-1) -- (3,3) -- (-1,3) -- (-1,-1);

\node[main][fill= blue] (1) at (0,0) {}; 
\node[rectangle,draw,inner sep=5][fill= red] (2) at (2,0) {}; 

\node[main][fill= blue] (3) at (2,2) {}; 
\node[rectangle,draw,inner sep=5][fill= red] (4) at (0,2) {}; 

\draw [line width=1.pt] (1) -- (4);
\draw [line width=1.pt] (2) -- (3);
\draw [line width=1.pt][->] (1) -- (3);
\draw [line width=1.pt][->] (2) -- (1);

\node[fill=none] at (4.5,1) (nodes) {$\game{1}{\cgstar}$~:};
\draw [line width=3.pt](6,-1) -- (10,-1) -- (10,3) -- (6,3) -- (6,-1);

\node[rectangle,draw,inner sep=5][fill= red] (1) at (7,0) {}; 
\node[main][fill= blue] (2) at (9,0) {}; 

\node[rectangle,draw,inner sep=5][fill= red] (3) at (9,2) {}; 
\node[main][fill= blue] (4) at (7,2) {}; 

\draw [line width=1.pt] (1) -- (4);
\draw [line width=1.pt] (2) -- (3);
\draw [line width=1.pt][->] (1) -- (3);
\draw [line width=1.pt][->] (2) -- (1);

\end{tikzpicture}
\end{center}

\vspace{0.75cm}
\begin{center}
\begin{tikzpicture}[node distance={15mm}, thick, main/.style = {draw, circle,}] 

\node[fill=none] at (-2.5,1) (nodes) {$\game{0}{-1}$~:};
\draw [line width=3.pt](-1,-1) -- (3,-1) -- (3,3) -- (-1,3) -- (-1,-1);

\node[rectangle,draw,inner sep=5][fill= red] (1) at (0,0) {}; 
\node[rectangle,draw,inner sep=5][fill= red] (2) at (2,0) {}; 
\node[main][fill= blue] (3) at (1,2) {};

\draw [line width=1.pt] (1) -- (3);
\draw [line width=1.pt] (2) -- (3);

\node[fill=none] at (4.5,1) (nodes) {$\game{1}{0}$~:};
\draw [line width=3.pt](6,-1) -- (10,-1) -- (10,3) -- (6,3) -- (6,-1);

\node[main][fill= blue] (1) at (7,0) {}; 
\node[main][fill= blue] (2) at (9,0) {}; 
\node[rectangle,draw,inner sep=5][fill= red] (3) at (8,2) {};

\draw [line width=1.pt] (1) -- (3);
\draw [line width=1.pt] (2) -- (3);

\end{tikzpicture}
\end{center}

\vspace{0.75cm}
\begin{center}
\begin{tikzpicture}[node distance={15mm}, thick, main/.style = {draw, circle,}] 

\node[fill=none] at (-3,1) (nodes) {$\cgdown$~:};
\draw [line width=3.pt](-1,-1) -- (3,-1) -- (3,3) -- (-1,3) -- (-1,-1);

\node[rectangle,draw,inner sep=5][fill= red] (1) at (0,0) {}; 
\node[main][fill= blue] (2) at (2,0) {}; 
\node[main][fill= blue] (3) at (2,2) {};
\node[rectangle,draw,inner sep=5][fill= red] (4) at (0,2) {};

\draw [line width=1.pt] (1) -- (2);
\draw [line width=1.pt][->] (1) -- (3);
\draw [line width=1.pt][->] (1) -- (4);
\draw [line width=1.pt] (3) -- (4);

\node[fill=none] at (4.5,1) (nodes) {$\cgup$~:};
\draw [line width=3.pt](6,-1) -- (10,-1) -- (10,3) -- (6,3) -- (6,-1);

\node[main][fill= blue] (1) at (7,0) {}; 
\node[rectangle,draw,inner sep=5][fill= red] (2) at (9,0) {}; 
\node[rectangle,draw,inner sep=5][fill= red] (3) at (9,2) {};
\node[main][fill= blue] (4) at (7,2) {};

\draw [line width=1.pt] (1) -- (2);
\draw [line width=1.pt][->] (1) -- (3);
\draw [line width=1.pt][->] (1) -- (4);
\draw [line width=1.pt] (3) -- (4);

\end{tikzpicture}
\end{center}

\vspace{0.75cm}
\begin{center}
\begin{tikzpicture}[node distance={15mm}, thick, main/.style = {draw, circle,}] 

\node[fill=none] at (-2.5,1) (nodes) {$\cgdown + \cgstar$~:};
\draw [line width=3.pt](-1,-1) -- (3,-1) -- (3,3) -- (-1,3) -- (-1,-1);

\node[main][fill= blue] (d1) at (0,0) {}; 
\node[rectangle,draw,inner sep=5][fill= red] (d2) at (2,0) {}; 
\node[rectangle,draw,inner sep=5][fill= red] (d3) at (1,2) {};

\draw [line width=1.pt] (d1) -- (d2);
\draw [line width=1.pt][->] (d1) -- (d3);
\draw [line width=1.pt][->] (d2) -- (d3);

\node[fill=none] at (4.5,1) (nodes) {$\cgup + \cgstar$~:};
\draw [line width=3.pt](6,-1) -- (10,-1) -- (10,3) -- (6,3) -- (6,-1);

\node[rectangle,draw,inner sep=5][fill= red] (d1) at (7,0) {}; 
\node[main][fill= blue] (d2) at (9,0) {}; 
\node[main][fill= blue] (d3) at (8,2) {};

\draw [line width=1.pt] (d1) -- (d2);
\draw [line width=1.pt][->] (d1) -- (d3);
\draw [line width=1.pt][->] (d2) -- (d3);

\end{tikzpicture}
\end{center}

\vspace{0.75cm}
\begin{center}
\begin{tikzpicture}[node distance={15mm}, thick, main/.style = {draw, circle,}] 

\node[fill=none] at (-2.5,1) (nodes) {$\game{0,\cgstar}{-1}$~:};
\draw [line width=3.pt](-1,-1) -- (3,-1) -- (3,3) -- (-1,3) -- (-1,-1);

\node[main][fill= blue] (1) at (0,0) {}; 
\node[rectangle,draw,inner sep=5][fill= red] (2) at (2,0) {}; 

\node[main][fill= blue] (3) at (2,2) {}; 
\node[rectangle,draw,inner sep=5][fill= red] (4) at (0,2) {}; 

\draw [line width=1.pt] (1) -- (4);
\draw [line width=1.pt] (2) -- (3);
\draw [line width=1.pt][->] (1) -- (3);
\draw [line width=1.pt] (2) -- (1);

\node[fill=none] at (4.5,1) (nodes) {$\game{1}{0,\cgstar}$~:};
\draw [line width=3.pt](6,-1) -- (10,-1) -- (10,3) -- (6,3) -- (6,-1);

\node[rectangle,draw,inner sep=5][fill= red] (1) at (7,0) {}; 
\node[main][fill= blue] (2) at (9,0) {}; 

\node[rectangle,draw,inner sep=5][fill= red] (3) at (9,2) {}; 
\node[main][fill= blue] (4) at (7,2) {}; 

\draw [line width=1.pt] (1) -- (4);
\draw [line width=1.pt] (2) -- (3);
\draw [line width=1.pt][->] (1) -- (3);
\draw [line width=1.pt] (2) -- (1);

\end{tikzpicture}
\end{center}

\vspace{0.75cm}
\begin{center}
\begin{tikzpicture}[node distance={15mm}, thick, main/.style = {draw, circle,}] 

\node[fill=none] at (-2.5,1) (nodes) {$0$~:};
\draw [line width=3.pt](-1,-1) -- (3,-1) -- (3,3) -- (-1,3) -- (-1,-1);

\node[fill=none] at (4.5,1) (nodes) {$\cgstar$~:};
\draw [line width=3.pt](6,-1) -- (10,-1) -- (10,3) -- (6,3) -- (6,-1);

\node[main][fill= blue] (1) at (8,2) {}; 
\node[rectangle,draw,inner sep=5][fill= red] (2) at (8,0) {}; 

\draw [line width=1.pt] (1) -- (2);
\end{tikzpicture}
\end{center}

\vspace{0.75cm}
\begin{center}
\begin{tikzpicture}[node distance={15mm}, thick, main/.style = {draw, circle,}] 

\node[fill=none] at (-2.5,1) (nodes) {$\nimber2$~:};
\draw [line width=3.pt](-1,-1) -- (3,-1) -- (3,3) -- (-1,3) -- (-1,-1);

\node[rectangle,draw,inner sep=5][fill= red] (1) at (0,0) {}; 
\node[main][fill= blue] (2) at (0,2) {}; 

\node[main][fill= blue] (3) at (2,0) {}; 
\node[rectangle,draw,inner sep=5][fill= red] (4) at (2,2) {}; 

\draw [line width=1.pt] (1) -- (2);
\draw [line width=1.pt] (3) -- (4);
\draw [line width=1.pt][->] (3) -- (1);
\draw [line width=1.pt][->] (3) -- (2);
\draw [line width=1.pt][->] (4) -- (1);
\draw [line width=1.pt][->] (4) -- (2);\node[fill=none] at (4.5,1) (nodes) {$\game{1}{-1}$~:};
\draw [line width=3.pt](6,-1) -- (10,-1) -- (10,3) -- (6,3) -- (6,-1);

\node[main][fill= blue] (1) at (7,0) {}; 
\node[main][fill= blue] (3) at (9,2) {}; 

\node[rectangle,draw,inner sep=5][fill= red] (2) at (9,0) {}; 
\node[rectangle,draw,inner sep=5][fill= red] (4) at (7,2) {}; 

\draw [line width=1.pt] (1) -- (2);
\draw [line width=1.pt] (2) -- (3);
\draw [line width=1.pt] (3) -- (4);
\draw [line width=1.pt] (4) -- (1);

\end{tikzpicture}
\end{center}

\end{document}